%% file: main.tex
\documentclass[hidelinks,onefignum,onetabnum]{siamart220329}

\input{ex_shared}

\ifpdf
\hypersetup{
	pdftitle={Energy Stability and Error Estimates for a Globally Divergence-Free DG Method of Arbitrary Order},
	pdfauthor={Yuchang Liu, Yan Jiang, and Mengping Zhang}
}
\fi

\allowdisplaybreaks
\usepackage{float}
\usepackage{mathrsfs}
\usepackage{xcolor}
\usepackage{etoolbox}
\usepackage{enumitem}
\usepackage{tikz}
\usetikzlibrary{arrows.meta,calc}
\tikzset{
	scheme/.style={font=\small,line width=.55pt,line cap=round,line join=round},
	meshline/.style={draw=black!55,line width=.45pt},
	active/.style={draw=black,line width=.9pt},
	direction/.style={-{Stealth[length=4pt,width=3pt]},line width=.65pt},
	stage/.style={draw,rounded corners=1.5pt,align=center,inner sep=5pt,
		font=\footnotesize},
	panel/.style={font=\footnotesize,align=center}
}

\AtBeginEnvironment{thebibliography}{\interlinepenalty=10000\relax}

\begin{document}
	
	\maketitle
	
	\begin{abstract}
		We propose and analyze an arbitrary-order globally divergence-free (GDF) discontinuous Galerkin (DG) method for the two-dimensional linear magnetic induction equation on unstructured triangular meshes. To the best of our knowledge, this is the first work to provide a rigorous energy-based analysis of high-order GDF-DG methods.
        First, we incorporate the edge normal-trace evolution into a closed, strictly local semidiscrete operator acting on the elementwise magnetic field alone.
        Then, we introduce two local modifications to the continuous electric field: a weighted patch moment correction that enforces local polynomial orthogonality and a vertex-patch stabilization that damps the remaining high-order modes.
        On shape-regular, quasi-uniform periodic triangulations, we prove semidiscrete energy stability with a sufficiently large mesh-independent penalty and establish a priori error estimates for smooth solutions.
        Numerical experiments support the analysis.
	\end{abstract}
	
	

\section{Introduction}\label{sec1}

We analyze arbitrary-order globally divergence-free (GDF) discontinuous Galerkin (DG) methods on unstructured triangular meshes for the magnetic induction equation
\begin{equation}\label{eq:intro_induction}
  \partial_t\mathbf B+\hat\nabla\times E(\mathbf u,\mathbf B)=0,
  \qquad \nabla\cdot\mathbf B=0,
\end{equation}
on a two-dimensional periodic domain $\Omega$. Throughout the paper, $\mathbf u=(u_x,u_y)^{\mathsf T}$ is independent of space and time, $U=|\mathbf u|$, and
\begin{equation}\label{eq:intro_electric_field}
  E(\mathbf u,\mathbf B)=u_yB_x-u_xB_y,
  \qquad \hat\nabla\times\phi=(\partial_y\phi,-\partial_x\phi)^{\mathsf T}.
\end{equation}
For vector fields, we write $\hat\nabla\times\mathbf w=\partial_xw_y-\partial_yw_x$. This subsystem of ideal magnetohydrodynamics (MHD) isolates magnetic transport from the nonlinear coupling of the full system. 

In particular, preserving the divergence-free constraint $\nabla\cdot\mathbf B=0$ is crucial in numerical simulations, as it prevents spurious magnetic monopoles and associated numerical instabilities~\cite{balsara2004second,li2005locally,li2012arbitrary,liu2025globally,liu2025arbitrarily}. Locally divergence-free DG bases~\cite{shu2009discontinuous,cockburn2004locally,li2005locally} enforce zero divergence within each cell but leave normal jumps across interfaces. Global divergence freedom requires an $H(\operatorname{div})$-conforming field~\cite{brezzi1985two,arnold2006finite,boffi2013mixed}; standard Galerkin formulations on this space couple shared degrees of freedom through a global mass matrix~\cite{cockburn2004locally,hu2017stable}. To retain a local, explicit update, Li and Xu~\cite{li2012arbitrary} evolved shared normal traces and reconstructed the interior field locally on overlapping Cartesian grids. This approach was developed further for non-overlapping Cartesian DG methods~\cite{fu2018globally,liu2025globally} and extended to unstructured triangular meshes~\cite{liu2025arbitrarily}.

Energy and error analyses are available for locally divergence-free or stabilized methods~\cite{besse2005convergence,heumann2013stabilized,sarkar2020priori} and globally coupled mixed formulations~\cite{hu2017stable,ding2026optimal}. For explicit reconstruction-based GDF-DG schemes, existing results include von Neumann analysis on uniform Cartesian grids~\cite{balsara2017von} and first-order stability and error analysis on overlapping Cartesian meshes~\cite{yang2016stability}. However, arbitrary-order analysis faces two obstacles: separate edge and cell evolutions obscure the underlying operator, and enforcing normal continuity introduces a reconstruction residual whose interior bubble components are not controlled by standard DG jump dissipation.

Compatible finite elements provide a useful viewpoint. Recently, Fan and Fu~\cite{fan2026structure} use a continuous electric field and an $H(\operatorname{div})$-conforming magnetic field in a structure-preserving splitting method for 2.5D ideal MHD on Cartesian meshes. Their discrete curl update preserves the divergence constraint; we use this de Rham structure~\cite{arnold2006finite,boffi2013mixed} to analyze strictly local reconstruction on triangles. \textit{To the best of our knowledge, the proposed framework provides the first rigorous energy-based analysis of high-order GDF-DG methods.}

We rewrite the edge-based method with physical $L^2$ reconstruction as a closed cellwise operator, separating the reconstruction residual from the DG operator without a global solve. Weighted patch moment corrections cancel its moments against common divergence-free patch polynomials, while vertex-patch stabilization controls the remaining components. For every $k\ge0$, we establish semidiscrete energy stability for a sufficiently large, mesh-independent penalty and, using a compatible local projection, an $L^2$ error bound of order $k+1/2$. Numerical experiments support our analysis.

\section{Proposed scheme}\label{sec2}
We reformulate the edge-based GDF update of~\cite{liu2025arbitrarily} as a closed cellwise operator and then introduce continuous electric-field corrections for stability.
\subsection{Triangulations, patches, and finite element spaces}\label{subsec:mesh_spaces}
	Let $\{\mathcal T_h\}$ be a family of conforming, uniformly shape-regular, and globally quasi-uniform periodic triangulations of $\Omega$ with edge set $\mathscr E$ and vertex set $\mathscr V$. For each closed triangle $K\in\mathcal T_h$ and edge $e\in\mathscr E$, set $h_K=\operatorname{diam}K$, $h_e=|e|$, and $h=\max_{K\in\mathcal T_h}h_K$, so that $ch\le h_K\le h$ with a constant $c>0$ independent of $h$. We assume that $h$ is sufficiently small so that every local patch
    and every fixed $\mathcal O(h)$ neighborhood used below embeds into
    $\mathbb R^2$ without periodic self-overlap.
	
	On each triangle $K\in\mathcal T_h$, let $\mathbf n_K$ and $\mathbf m_K=(-n_{K,y},n_{K,x})^{\mathsf T}$ be the outward unit normal and counterclockwise unit tangent along $\partial K$. On each shared edge $e=\partial K^-\cap\partial K^+\in\mathscr E$, we fix a unit normal $\mathbf n_e$ pointing from $K^-$ to $K^+$ and set $\mathbf m_e=(-n_{e,y},n_{e,x})^{\mathsf T}$, oriented from endpoint $v_a$ to endpoint $v_b$. The jump and average of a piecewise smooth function $w$ across $e$ are
	\[
	[w]_e=w^+-w^-,\qquad \{w\}_e=\frac12(w^-+w^+),\qquad w^\pm=(w|_{K^\pm})|_e.
	\]
	For each vertex $v\in\mathscr V$, let $\omega_v=\bigcup_{K\ni v}K$ denote its vertex patch comprising $m_v$ incident triangles, and $\lambda_v$ the continuous $\mathbb P^1$ hat function equal to $1$ at $v$ and $0$ at all other vertices. For each edge $e=(v_a,v_b)\in\mathscr E$ shared by $K^-$ and $K^+$, we define the two-cell edge patch and quadratic edge weight
	\begin{equation}\label{eq:scheme_edge_weight}
		\omega_e=K^-\cup K^+,\qquad\lambda_e=\lambda_{v_a}\lambda_{v_b}.
	\end{equation}
	Both $\lambda_v$ and $\lambda_e$ vanish on the outer boundaries $\partial\omega_v$ and $\partial\omega_e$, respectively.

	Let $\mathbb P^j(D)$ denote the space of polynomials of degree at most $j$ on $D\subset\mathbb R^2$, with the convention $\mathbb P^j(D)=\{0\}$ whenever $j<0$. When $D$ is a multi-cell patch such as $\omega_v$ or $\omega_e$, $\mathbb P^j(D)$ denotes restrictions of single polynomials on $D$ rather than broken cellwise polynomials. For a fixed integer $k\ge0$, we define the broken polynomial space, the globally divergence-free DG space, and the continuous Lagrange space by
	\begin{align}
		\mathbf V_h^k&=\{\mathbf w\in[L^2(\Omega)]^2:
		\mathbf w|_K\in[\mathbb P^k(K)]^2,\ \forall K\in\mathcal T_h\},\label{eq:alt_V}\\
		\mathbf W_h^k&=\{\mathbf w\in\mathbf V_h^k:
		\nabla\cdot\mathbf w|_K=0,\ \forall K\in\mathcal T_h,
		\ [\mathbf w\cdot\mathbf n_e]_e=0,\ \forall e\in\mathscr E\},\label{eq:alt_W}\\
		S_h^{k+1}&=\{\phi\in C^0(\Omega):
		\phi|_K\in\mathbb P^{k+1}(K),\ \forall K\in\mathcal T_h\}.
		\label{eq:alt_S}
	\end{align}
	Here, $\mathbf W_h^k$ is linked to the continuous stream-function space $S_h^{k+1}$ through the discrete polynomial de Rham complex~\cite[Section~3.5]{arnold2006finite}, and coincides with the divergence-free subspace of the classical $H(\operatorname{div};\Omega)$-conforming Brezzi--Douglas--Marini space $\mathrm{BDM}_k(\mathcal T_h)$ for $k\ge1$ or the Raviart--Thomas space $\mathrm{RT}_0(\mathcal T_h)$ for $k=0$~\cite{brezzi1985two, boffi2013mixed}. 
    
    For a triangle $K$, let $\beta_K=\prod_{v\in K}\lambda_v|_K\in\mathbb P^3(K)$ be the cubic element bubble on $K$, and define the scalar and divergence-free bubble spaces
	\begin{equation}\label{eq:scheme_bubbles}
		\mathcal B_K^{k+1}=\beta_K\mathbb P^{k-2}(K),
		\qquad \mathbf Z_K^k=\hat\nabla\times\mathcal B_K^{k+1}.
	\end{equation}
	By the local polynomial de Rham exact sequence on $K$~\cite{arnold2006finite}, $\hat\nabla\times:\mathcal B_K^{k+1}\to\mathbf Z_K^k$ is an isomorphism onto the zero-normal-trace divergence-free polynomial space:
		\begin{equation}\label{eq:scheme_zero_normal_bubbles}
			\mathbf Z_K^k=\left\{\mathbf z\in[\mathbb P^k(K)]^2:\nabla\cdot\mathbf z=0\text{ in }K,\ \mathbf z\cdot\mathbf n_K=0\text{ on }\partial K\right\}.
		\end{equation}
	We denote by $Q_K:[L^2(K)]^2\to\mathbf Z_K^k$ the physical $L^2$-orthogonal projection onto $\mathbf Z_K^k$:
	\begin{equation}\label{eq:scheme_Q}
		(Q_K\mathbf w,\mathbf z)_K=(\mathbf w,\mathbf z)_K,
		\qquad\forall\mathbf z\in\mathbf Z_K^k.
	\end{equation}
	Throughout the paper, $(\cdot,\cdot)_K$, $(\cdot,\cdot)_e$, and $(\cdot,\cdot)_\omega$ denote the standard $L^2$ inner products on a cell $K$, an edge $e$, and a patch $\omega$, respectively, with the subscript omitted over $\Omega$.

\subsection{An equivalent formulation of the triangular GDF method}\label{subsec:scheme_equivalence}

For $\mathbf B_h\in\mathbf W_h^k$, set $E_K=E(\mathbf u,\mathbf B_h|_K)$.
Following~\cite{liu2025arbitrarily}, define the edge Lax--Friedrichs (LF) and vertex electric fields by
\begin{align}
 \widehat E_e&=\{E\}_e+\frac{U}{2}[\mathbf B_h\cdot\mathbf m_e]_e,
\label{eq:scheme_LFfield}\\
 \widetilde E_v&=\frac1{m_v}\sum_{K\ni v}E_K(v)
 +\frac{U}{m_v}\sum_{e\ni v}[\mathbf B_h\cdot\mathbf m_e]_e(v).
 \label{eq:scheme_vertexE}
\end{align}
The standard DG operator $G_h(\mathbf B_h)\in\mathbf V_h^k$ is defined by
\begin{equation}\label{eq:DG}
 (G_h(\mathbf B_h),\mathbf w)_K
 =\int_K E_K(\partial_yw_x-\partial_xw_y)\,\mathrm dK
 +\sum_{e\subset\partial K}(\widehat E_e,\mathbf w\cdot\mathbf m_K)_e,
\end{equation}
for all $\mathbf w\in[\mathbb P^k(K)]^2$. Define $E_h^0\in S_h^{k+1}$ by its edge degrees of freedom
\begin{equation}\label{eq:scheme_Etrace}
 \begin{aligned}
 E_h^0|_e(v_a)&=\widetilde E_{v_a},\qquad E_h^0|_e(v_b)=\widetilde E_{v_b},\\
 (E_h^0,q)_e&=(\widehat E_e,q)_e,\qquad\forall q\in\mathbb P^{k-1}(e),
 \end{aligned}
\end{equation}
and extend it arbitrarily into each cell. We define the baseline semidiscrete scheme by
\begin{equation}\label{eq:alt_baseline_rate}
 \begin{aligned}
 \frac{\mathrm d\mathbf B_h}{\mathrm dt}
 &=\mathcal L_h^0(\mathbf B_h),\\
 \mathcal L_h^0(\mathbf B_h)|_K
 &=-\hat\nabla\times E_h^0|_K
 +Q_K\bigl(G_h(\mathbf B_h)|_K+\hat\nabla\times E_h^0|_K\bigr).
 \end{aligned}
\end{equation}
This definition is independent of the interior extensions: for any two extensions, their difference
$\eta_K\in\mathcal B_K^{k+1}$ satisfies
$-\hat\nabla\times\eta_K+Q_K(\hat\nabla\times\eta_K)=\mathbf0$.

Since $Q_K$ has range $\mathbf Z_K^k$, \eqref{eq:alt_baseline_rate} gives
\begin{equation}\label{eq:alt_trace_bubble}
 \begin{aligned}
 \mathcal L_h^0(\mathbf B_h)\cdot\mathbf n_e
 &=-\partial_{\mathbf m_e}(E_h^0|_e),
 &&\forall e\in\mathscr E,\\
 Q_K\bigl(\mathcal L_h^0(\mathbf B_h)|_K\bigr)
 &=Q_K\bigl(G_h(\mathbf B_h)|_K\bigr),
 &&\forall K\in\mathcal T_h.
 \end{aligned}
\end{equation}
In particular, $\mathcal L_h^0$ maps $\mathbf W_h^k$ into itself.
For the forward Euler step
$\mathbf B_h^{n+1}=\mathbf B_h^n+\Delta t\,\mathcal L_h^0(\mathbf B_h^n)$,
set $b_e^n=\mathbf B_h^n\cdot\mathbf n_e$.
The first identity in~\eqref{eq:alt_trace_bubble}, integrated by parts
using~\eqref{eq:scheme_Etrace}, yields
\begin{equation}\label{eq:alt_edge0}
 (b_e^{n+1}-b_e^n,q)_e=\Delta t\bigl[
 (\widehat E_e^n,\partial_{\mathbf m_e}q)_e
 -\widetilde E_{v_b}^nq(v_b)+\widetilde E_{v_a}^nq(v_a)\bigr],\quad \forall q\in\mathbb P^k(e).
\end{equation}
The second identity gives the interior bubble relations
\begin{equation}\label{eq:alt_bubble_update}
 (\mathbf B_h^{n+1}-\mathbf B_h^n
 -\Delta t\,G_h(\mathbf B_h^n),\mathbf z)_K=0,
 \qquad\forall\mathbf z\in\mathbf Z_K^k.
\end{equation}
For the updated outward traces
$b_{K,e}^{n+1}=(\mathbf n_K\cdot\mathbf n_e)b_e^{n+1}$, define
\begin{equation}\label{eq:alt_affine}
 \mathcal A_K(b)=\{\mathbf w\in[\mathbb P^k(K)]^2:
 \nabla\cdot\mathbf w=0\text{ in }K,\ 
 \mathbf w\cdot\mathbf n_K=b_{K,e}\text{ on each }e\subset\partial K\}.
\end{equation}
It is clear that the direct update $\mathbf B_h^{n+1}|_K$ belongs to $\mathcal A_K(b_K^{n+1})$.
By~\eqref{eq:scheme_zero_normal_bubbles}, any two fields in this set
differ by an element of $\mathbf Z_K^k$. Thus~\eqref{eq:alt_bubble_update} implies that $\mathbf B_h^{n+1}|_K$ is the unique minimizer:
\begin{equation}\label{eq:alt_baseline_algorithm}
 \mathbf B_h^{n+1}|_K
 =\underset{\mathbf w\in\mathcal A_K(b_K^{n+1})}{\operatorname{argmin}}
 \frac12\|\mathbf w-\mathbf B_h^n|_K-\Delta t\,G_h(\mathbf B_h^n)|_K\|_K^2.
\end{equation}
Hence, the baseline formulation~\eqref{eq:alt_baseline_rate}
recovers the edge-based update of~\cite{liu2025arbitrarily},
with the reconstruction replaced by constrained least-squares
minimization in the physical $L^2(K)$ norm. In particular, all terms in~\eqref{eq:alt_baseline_rate} depend only
on the current magnetic field over a local stencil, so no global
mass-matrix solve is required. It also conserves $\int_\Omega\mathbf B_h\,\mathrm d\Omega$: the curl of a periodic potential
has zero global mean, and every field in $\mathbf Z_K^k$ has zero cell mean.

By the equivalent formulation \eqref{eq:alt_baseline_rate}, we can see the exact reconstruction residual is \begin{equation}\label{eq:alt_baseline_residual}
 \mathcal L_h^0(\mathbf B_h)|_K-G_h(\mathbf B_h)|_K
 =(I-Q_K)\bigl(-\hat\nabla\times E_h^0|_K-G_h(\mathbf B_h)|_K\bigr).
\end{equation}
The DG jump dissipation alone does not control this residual, motivating the corrections below.

\subsection{Electric-field corrections and stabilization}\label{subsec:alt_proposed}
We modify only the exterior curl in~\eqref{eq:alt_baseline_rate}, setting
\begin{equation}\label{eq:scheme_electric_field}
 E_h=E_h^0+E_h^{\rm mom}+E_h^{\rm stab},\qquad E_h^{\mathrm{mom}}, E_h^{\mathrm{stab}}\in S_h^{k+1},
\end{equation}
and obtain the proposed semidiscrete scheme
\begin{equation}\label{eq:scheme_final_rate}
\begin{aligned}
 \frac{\mathrm d\mathbf B_h}{\mathrm dt}&=\mathcal L_h(\mathbf B_h),\\
 \mathcal L_h(\mathbf B_h)|_K
 &=-\hat\nabla\times E_h|_K
 +Q_K\bigl(G_h(\mathbf B_h)|_K+\hat\nabla\times E_h^0|_K\bigr).
 \end{aligned}
\end{equation}
The moment correction makes the local reconstruction residuals orthogonal
to degree-$k$ divergence-free polynomials; the patch stabilization damps the
remaining modes. 
\subsubsection{Continuous moment corrections}
	By integration by parts on $K$, the reconstruction residual is driven by the boundary discrepancies between the cellwise convective electric field $E_K=E(\mathbf u,\mathbf B_h|_K)$ and the baseline traces $\widehat E_e$ and $\widetilde E_v$:
	\begin{equation}\label{eq:scheme_defect_data}
		d_{K,e}=E_K|_e-\widehat E_e,
		\qquad z_{K,v}=E_K(v)-E_h^0(v).
	\end{equation}
	On each edge $e=(v_a,v_b)\subset\partial K$, we split $d_{K,e}$ into its linear endpoint interpolant and a remainder vanishing at $v_a,v_b$:
	\begin{equation}\label{eq:scheme_defect_split}
		d_{K,e}=d_{K,e}(v_a)\lambda_{v_a}|_e
		+d_{K,e}(v_b)\lambda_{v_b}|_e+r_{K,e},
		\qquad r_{K,e}(v_a)=r_{K,e}(v_b)=0.
	\end{equation}

Note that $d_{K,e}(v)\ne z_{K,v}$ in general. For each vertex $v$ of $K$, we define an auxiliary polynomial $\Theta_{K,v}\in\mathbb P^{k+1}(K)$ by
	\begin{align}
		\Theta_{K,v}(v)
		&= z_{K,v},\qquad
		\Theta_{K,v}(v')=0,\quad \forall v'\ne v,\\
		(\Theta_{K,v},q)_f
		&=
		\begin{cases}
			(d_{K,f}(v)\lambda_v,q)_f,&v\in f,\\
			0,&v\notin f,
		\end{cases}
		\qquad \forall q\in\mathbb P^{k-1}(f),\quad
		\forall f\subset\partial K,
		\label{eq:scheme_theta_vertex_trace}\\
		(\nabla\Theta_{K,v},\nabla b)_K
		&=\sum_{\substack{f\subset\partial K\\v\in f}}
		(d_{K,f}(v)\lambda_v,\partial_{\mathbf n_K}b)_f,
		\qquad \forall b\in\mathcal B_K^{k+1}.
		\label{eq:scheme_theta_vertex_bubble}
	\end{align}
	
	Similarly, for each edge $e\subset\partial K$, we define $\Theta_{K,e}\in\mathbb P^{k+1}(K)$ vanishing at all three vertices of $K$ and satisfying
	\begin{align}
		(\Theta_{K,e},q)_f
		&=\begin{cases}(r_{K,e},q)_e,&f=e,\\0,&f\ne e,\end{cases}
		\qquad \forall q\in\mathbb P^{k-1}(f),\quad \forall f\subset\partial K,
		\label{eq:scheme_theta_edge_trace}\\
		(\nabla\Theta_{K,e},\nabla b)_K
		&=(r_{K,e},\partial_{\mathbf n_K}b)_e,
		\qquad \forall b\in\mathcal B_K^{k+1}.
		\label{eq:scheme_theta_edge_bubble}
	\end{align}
	For $k\ge2$, positive definiteness of $(\nabla\cdot,\nabla\cdot)_K$ on $\mathcal B_K^{k+1}$ uniquely fixes the interior bubble components.
	
	Following the weighted local projection stabilization framework~\cite[Section~2.1]{matthies2007unified}, we project these discontinuous potentials onto continuous patch-supported fields matching their $\mathbb P^{k-1}$ moments. For each $v\in\mathscr V$ and $e=(v_a,v_b)\in\mathscr E$, we uniquely determine single patch polynomials $\psi_v\in\mathbb P^{k-1}(\omega_v)$ and $\psi_e\in\mathbb P^{k-1}(\omega_e)$ from the positive-definite weighted Gram systems
	\begin{align}
		(\lambda_v\psi_v,q)_{\omega_v}
		&=\sum_{K\ni v}(\Theta_{K,v},q)_K,
		&&\forall q\in\mathbb P^{k-1}(\omega_v),\label{eq:scheme_vertex_moments}\\
		(\lambda_e\psi_e,q)_{\omega_e}
		&=\sum_{K\in\{K^-,K^+\}}(\Theta_{K,e},q)_K,
		&&\forall q\in\mathbb P^{k-1}(\omega_e),\label{eq:scheme_edge_moments}
	\end{align}
	and assemble the continuous moment correction
	\begin{equation}\label{eq:scheme_Xi1}
		E_h^{\rm mom}=
		\sum_{v\in\mathscr V}\lambda_v\psi_v
		+\sum_{e\in\mathscr E}\lambda_e\psi_e\in S_h^{k+1}.
	\end{equation}

For $\mathbf w\in\mathbf W_h^k$, use the above construction with
$\mathbf B_h=\mathbf w$ and define the moment-corrected reconstruction residual
\begin{equation}\label{eq:en_defect}
		\mathcal R_h(\mathbf w)
		=\mathcal L_h^0(\mathbf w)
		-\hat\nabla\times E_h^{\rm mom}(\mathbf w)-G_h(\mathbf w).
	\end{equation}
For arguments common to vertex and edge patches, we use $a\in\mathscr V\sqcup\mathscr E$ to denote either a vertex $v$ or an edge $e$,
with patch quantities indexed accordingly; sums over $a$ run over this set.
For any patch $\omega_a$, denote the space of single divergence-free polynomials of degree at most $k$ on $\omega_a$ by
	\[
	\mathbf W^k(\omega_a)=\bigl\{\mathbf p\in[\mathbb P^k(\omega_a)]^2:\nabla\cdot\mathbf p=0\bigr\}.
	\]
The moment correction yields the following local decomposition and polynomial orthogonality.
\begin{lemma}\label{lem:en_sources}
For every $\mathbf w\in\mathbf W_h^k$, there exist
$\mathbf R_a(\mathbf w)\in\mathbf V_h^k$ with
$\operatorname{supp}\mathbf R_a(\mathbf w)\subset\omega_a$ such that
\begin{equation}\label{eq:en_source_decomposition}
 \mathcal R_h(\mathbf w)=\sum_{v\in\mathscr V}\mathbf R_v(\mathbf w)+\sum_{e\in\mathscr E}\mathbf R_e(\mathbf w),
\end{equation}
and
\begin{equation}\label{eq:en_source_orth}
 (\mathbf R_a(\mathbf w),\mathbf p)_{\omega_a}=0,
 \qquad\forall\mathbf p\in\mathbf W^k(\omega_a),\quad
 a\in\mathscr V\sqcup\mathscr E.
\end{equation}
\end{lemma}
\begin{proof}
Assemble $\Theta_a|_K=\Theta_{K,a}$ on $\omega_a$ and extend by zero.
Define $\mathbf S_v,\mathbf S_e\in\mathbf V_h^k$, supported in
$\omega_v,\omega_e$, respectively, by
\begin{equation}\label{eq:en_boundary_liftings}
 \begin{aligned}
 (\mathbf S_v,\mathbf q)_K
 &=-\sum_{\substack{f\subset\partial K\\v\in f}}
 (d_{K,f}(v)\lambda_v,\mathbf q\cdot\mathbf m_K)_f,\qquad\forall\mathbf q\in[\mathbb P^k(K)]^2\\
 (\mathbf S_e,\mathbf q)_K
 &=-(r_{K,e},\mathbf q\cdot\mathbf m_K)_e,
 \qquad\forall\mathbf q\in[\mathbb P^k(K)]^2,
 \end{aligned}
\end{equation}
on cells $K\ni v$ and $K\supset e$, respectively. Define $\mathbf F_a$ cellwise by
$\mathbf F_a|_K=\hat\nabla\times(\Theta_a|_K)-\mathbf S_a|_K$
for every $K\in\mathcal T_h$, and set
\begin{equation}\label{eq:en_local_contribution}
 \mathbf R_a(\mathbf w)=\mathbf F_a-\hat\nabla\times(\lambda_a\psi_a).
\end{equation}
The bubble conditions~\eqref{eq:scheme_theta_vertex_bubble}
and~\eqref{eq:scheme_theta_edge_bubble} give
$Q_K\mathbf F_a=\mathbf0$.
The vertex and edge data, together with
\eqref{eq:scheme_defect_split} and~\eqref{eq:DG}, give, for some
$b_K\in\mathcal B_K^{k+1}$,
\begin{equation}\label{eq:en_potential_lifting_sums}
 \begin{aligned}
 \sum_a\Theta_a|_K&=E_K-E_h^0|_K+b_K,\\
 \sum_a\mathbf S_a|_K&=G_h(\mathbf w)|_K+\hat\nabla\times E_K.
 \end{aligned}
\end{equation}
Since $\hat\nabla\times b_K\in\mathbf Z_K^k$, taking the rotated
gradient of the first identity, subtracting the second, and
applying $I-Q_K$ yields
\begin{equation}\label{eq:en_source_sum}
 \begin{aligned}
 \sum_a\mathbf F_a|_K
 &=(I-Q_K)(-\hat\nabla\times E_h^0|_K-G_h(\mathbf w)|_K)\\
 &=\mathcal L_h^0(\mathbf w)|_K-G_h(\mathbf w)|_K.
 \end{aligned}
\end{equation}
The last equality follows from~\eqref{eq:alt_baseline_rate}.
Summing~\eqref{eq:en_local_contribution} and using
\eqref{eq:scheme_Xi1} and~\eqref{eq:en_defect} proves~\eqref{eq:en_source_decomposition}.

For orthogonality, the defining vertex values and edge moments in
\eqref{eq:scheme_theta_vertex_trace} and
\eqref{eq:scheme_theta_edge_trace} imply that $\Theta_a$ has zero
trace on $\partial\omega_a$ and satisfies the following identities
on the interior patch edges:
\begin{equation}\label{eq:en_trace_identity}
 \begin{aligned}
 \Theta_v^-|_f-\Theta_v^+|_f
 &=\bigl(d_{K^-,f}(v)-d_{K^+,f}(v)\bigr)\lambda_v|_f,
 &&f\ni v,\\
 \Theta_e^-|_e-\Theta_e^+|_e
 &=r_{K^-,e}-r_{K^+,e}.
 \end{aligned}
\end{equation}
For $\mathbf p\in\mathbf W^k(\omega_a)$, the zero boundary trace
of $\Theta_a$, \eqref{eq:en_trace_identity}, and
\eqref{eq:en_boundary_liftings} give
\[
\sum_{K\subset\omega_a}
(\Theta_a,\mathbf p\cdot\mathbf m_K)_{\partial K}
=-(\mathbf S_a,\mathbf p)_{\omega_a}.
\]
Elementwise integration by parts therefore yields
\begin{equation}\label{eq:en_source_green}
 \begin{aligned}
 (\mathbf F_a,\mathbf p)_{\omega_a}
 &=(\Theta_a,\hat\nabla\times\mathbf p)_{\omega_a}
 -\sum_{K\subset\omega_a}
  (\Theta_a,\mathbf p\cdot\mathbf m_K)_{\partial K}
 -(\mathbf S_a,\mathbf p)_{\omega_a}\\
 &=(\Theta_a,\hat\nabla\times\mathbf p)_{\omega_a},\\
 (\mathbf R_a(\mathbf w),\mathbf p)_{\omega_a}
 &=(\Theta_a-\lambda_a\psi_a,
   \hat\nabla\times\mathbf p)_{\omega_a}=0,
 \end{aligned}
\end{equation}
where the last equality follows from the zero boundary trace of
$\lambda_a\psi_a$, the moment conditions
\eqref{eq:scheme_vertex_moments}--\eqref{eq:scheme_edge_moments},
and $\hat\nabla\times\mathbf p\in\mathbb P^{k-1}(\omega_a)$.
This proves~\eqref{eq:en_source_orth}.
\end{proof}

	\subsubsection{Stabilization on vertex patches}\label{subsec:scheme_stabilization}
	Set 
	\begin{equation}\label{eq:scheme_patch_size}
		h_v=\left(\frac{2}{m_v}\sum_{K\ni v}|K|\right)^{1/2},
	\end{equation}
	and introduce the local $H_0^1(\omega_v)$ finite element space together with its $\mathbb P^{k-1}(\omega_v)$-orthogonal subspace:
	\begin{align}
		S_0^{k+1}(\omega_v)&=\{\phi\in S_h^{k+1}:\operatorname{supp}\phi\subset\omega_v\},
		\label{eq:scheme_XP}\\
		H_v&=\{\phi\in S_0^{k+1}(\omega_v):(\phi,q)_{\omega_v}=0,
		\ \forall q\in\mathbb P^{k-1}(\omega_v)\}.
		\label{eq:scheme_HP}
	\end{align}
	On each vertex patch $\omega_v$, for $\mathbf w\in \mathbf W_h^k$, we define $\zeta_v(\mathbf w)\in H_v$ by solving
	\begin{equation}\label{eq:scheme_zeta}
		(\nabla\zeta_v(\mathbf w),\nabla\phi)_{\omega_v}
		=(\mathbf w,\hat\nabla\times\phi)_{\omega_v},
		\qquad \forall\phi\in H_v,
	\end{equation}
	and define, for a dimensionless penalty parameter $\tau\ge0$ independent of $h$,
	\begin{equation}\label{eq:scheme_finalE}
		E_h^{\rm stab}(\mathbf w)=\tau\sum_{v\in\mathscr V}\frac{U}{h_v}\zeta_v(\mathbf w)\in S_h^{k+1}.
	\end{equation}

With $\mathcal P_v$ denoting the $L^2$-orthogonal projection onto
$\hat\nabla\times H_v$, the local problem~\eqref{eq:scheme_zeta} gives
$\mathcal P_v\mathbf w=\hat\nabla\times\zeta_v(\mathbf w)$ and
\begin{equation}\label{eq:scheme_stab_curl}
 \hat\nabla\times E_h^{\rm stab}(\mathbf w)
 =\tau\sum_{v\in\mathscr V}\frac{U}{h_v}\mathcal P_v\mathbf w.
\end{equation}
Testing the stabilization contribution
$-\hat\nabla\times E_h^{\rm stab}(\mathbf w)$ with $\mathbf w$, we find that each
patch contributes the nonpositive energy term
$-\tau(U/h_v)\|\mathcal P_v\mathbf w\|_{\omega_v}^2$; 
this is the local projection stabilization mechanism
of~\cite{codina2000stabilization,matthies2007unified,luo2026local}.

In particular, if $\mathbf w|_{\omega_v}=(w_x,w_y)^{\mathsf T}\in[\mathbb P^k(\omega_v)]^2$,
integration by parts and the moments defining $H_v$ give
\begin{equation}\label{eq:scheme_polynomial_consistency}
 (\mathbf w,\hat\nabla\times\phi)_{\omega_v}
 =(\partial_xw_y-\partial_yw_x,\phi)_{\omega_v}=0,
 \qquad\forall\phi\in H_v,
\end{equation}
so $\zeta_v(\mathbf w)=0$.

Moreover, a crucial property of this vertex stabilization is the following patch regularity. It implies that \textit{the vertex-patch projections $\mathcal P_v\mathbf w$ provide global control
of the local best approximation errors of $\mathbf w$ in $\mathbf W^k(\omega_v)$.}  This provides a key estimate in the energy analysis. 
The proof is highly technical and uses local polynomial multipliers, a mean-zero load
reconstruction, and a Poisson--Ritz approximation. Details are
given in Appendix~\ref{app:patch_regularity}.

\begin{lemma}[Patch regularity]\label{lem:en_regularity}
There exists a constant $C>0$, depending only on $k$, $\Omega$, and the mesh regularity bounds, such that for all $\mathbf w\in\mathbf W_h^k$,
\begin{equation}\label{eq:en_regularity}
 \sum_{v\in\mathscr V}\|\mathcal P_v\mathbf w\|_{\omega_v}^2
 \le\sum_{v\in\mathscr V}\inf_{\mathbf p\in\mathbf W^k(\omega_v)}\|\mathbf w-\mathbf p\|_{\omega_v}^2
 \le C\sum_{v\in\mathscr V}\|\mathcal P_v\mathbf w\|_{\omega_v}^2.
\end{equation}
\end{lemma}

\subsection{Connection with compatible finite element methods}\label{subsec:compatible_structure}

On the torus, the discrete de Rham structure gives every
$\mathbf w\in\mathbf W_h^k$ a continuous potential representation
\begin{equation}\label{eq:compatible_potential}
 \mathbf w=\overline{\mathbf w}+\hat\nabla\times\phi_h,
 \qquad\overline{\mathbf w}=\frac1{|\Omega|}\int_\Omega\mathbf w\,\mathrm d\Omega,
 \qquad\phi_h\in S_h^{k+1},\quad\int_\Omega\phi_h\,\mathrm d\Omega=0.
\end{equation}
The constant part carries the periodic fluxes. For $\eta\in H_v$, the zero boundary trace gives
$(\overline{\mathbf w},\hat\nabla\times\eta)_{\omega_v}=0$, so
\eqref{eq:scheme_zeta} becomes
\begin{equation}\label{eq:compatible_patch_projection}
 (\nabla\zeta_v(\mathbf w),\nabla\eta)_{\omega_v}
 =(\nabla\phi_h,\nabla\eta)_{\omega_v},\qquad\forall\eta\in H_v.
\end{equation}
Hence $\zeta_v(\mathbf w)$ is the local energy projection of $\phi_h$ onto $H_v$:
the stabilization damps local energy components of the same continuous potential.

\section{Energy stability}\label{sec:energy_draft}
	We establish the semidiscrete energy stability of~\eqref{eq:scheme_final_rate}  for any fixed polynomial degree $k\ge0$. Throughout Sections~\ref{sec:energy_draft} and~\ref{sec:error}, generic constants $C>0$ depend only on $k$, $\Omega$, and the mesh regularity bounds.
	
	\subsection{Main result and sketch of the proof}\label{subsubsec:en_constant_result}
	\begin{theorem}[Energy stability]\label{thm:en_constant}
		There exists a constant $C_R>0$ such that whenever $\tau\ge C_R^2/2$, every initial datum $\mathbf B_h(0)\in\mathbf W_h^k$ determines a unique global semidiscrete solution $\mathbf B_h(t)\in\mathbf W_h^k$ of~\eqref{eq:scheme_final_rate} satisfying
		\begin{equation}\label{eq:en_main_constant}
			\frac12\frac{\mathrm d}{\mathrm dt}\|\mathbf B_h(t)\|^2\le0,
			\qquad
			\|\mathbf B_h(t)\|\le\|\mathbf B_h(0)\|,
			\qquad\forall t\ge0.
		\end{equation}
	\end{theorem}
	
	The proof of Theorem~\ref{thm:en_constant} rests on decomposing the spatial operator $\mathcal L_h$ into three parts: the dissipative DG predictor $G_h$, the moment-corrected reconstruction residual $\mathcal R_h$ defined in~\eqref{eq:en_defect}, and the vertex-patch stabilization. By~\eqref{eq:scheme_stab_curl}, for $\mathbf w\in\mathbf W_h^k$ we have
	\begin{equation}\label{eq:en_decomposition}
		\mathcal L_h(\mathbf w)
		=G_h(\mathbf w)+\mathcal R_h(\mathbf w)
		-\tau\sum_{v\in\mathscr V}\frac{U}{h_v}\mathcal P_v\mathbf w.
	\end{equation}
	For clarity, we outline the proof using two seminorms. For $\mathbf w\in\mathbf W_h^k$, define the interface jump dissipation seminorm
	\begin{equation}\label{eq:en_jump}
		J_U(\mathbf w)
		=\left(U\sum_{e\in\mathscr E}
		\|[\mathbf w\cdot\mathbf m_e]_e\|_e^2\right)^{1/2}.
	\end{equation}
	With the global patch stabilization seminorm
	\begin{equation}\label{eq:en_patch_projection}
		D_U(\mathbf w)=\left(\sum_{v\in\mathscr V}
		\frac{U}{h_v}\|\mathcal P_v\mathbf w\|_{\omega_v}^2\right)^{1/2},
	\end{equation}
	projection orthogonality yields
	\begin{equation}\label{eq:en_patch_identity}
		(\mathcal P_v\mathbf w,\mathbf w)_{\omega_v}
		=\|\mathcal P_v\mathbf w\|_{\omega_v}^2
		=\|\nabla\zeta_v(\mathbf w)\|_{\omega_v}^2.
	\end{equation}

\textit{Sketch of the proof.} Since $\mathcal L_h$ is linear on the finite-dimensional space $\mathbf W_h^k$, the semidiscrete system has a unique global solution for every initial datum. Testing~\eqref{eq:en_decomposition} with $\mathbf w=\mathbf B_h$ yields three terms
	\begin{align*}\frac{1}{2}\frac{\mathrm d}{\mathrm dt}\left\| \mathbf B_h \right\|^2&=(\mathcal L_h(\mathbf B_h),\mathbf B_h)
		\\&=(G_h(\mathbf B_h),\mathbf B_h)+(\mathcal R_h(\mathbf B_h),\mathbf B_h)
		-\tau\sum_{v\in\mathscr V}\frac{U}{h_v}(\mathcal P_v\mathbf B_h,\mathbf B_h). \end{align*}
	Using~\eqref{eq:en_patch_projection} and~\eqref{eq:en_patch_identity}, the third term is
		\begin{equation}\label{eq:en_stabilization_dissipation}
			-\tau\sum\limits_{v\in\mathscr V} \frac{U}{h_v}(\mathcal P_v\mathbf B_h,\mathbf B_h)= -\tau D_U(\mathbf B_h)^2.
		\end{equation}
	For the first term, integrating \eqref{eq:DG} by parts and summing over $K\in\mathcal T_h$ gives
	\begin{equation}\label{eq:energy_dg}
		(G_h(\mathbf B_h),\mathbf B_h)=-\frac{1}{2} J_U(\mathbf B_h)^2. 
	\end{equation}
	Both~\eqref{eq:en_stabilization_dissipation} and~\eqref{eq:energy_dg} dissipate the $L^2$ energy. The main difficulty lies in estimating the cross term $(\mathcal R_h(\mathbf B_h),\mathbf B_h)$. We will prove in~\cref{lem:en_defect_bound} that
	\begin{equation} \label{eq:en_defect_bound}\left|(\mathcal R_h(\mathbf B_h),\mathbf B_h)\right|\le C_R J_U(\mathbf B_h)D_U(\mathbf B_h). \end{equation}
	For $\tau\ge C_R^2/2$, combining~\eqref{eq:en_stabilization_dissipation}, \eqref{eq:energy_dg}, and~\eqref{eq:en_defect_bound} gives
	\begin{equation}\label{eq:en_constant_energy}
		\begin{aligned}
			\frac12\frac{\mathrm d}{\mathrm dt}\|\mathbf B_h\|^2&\le -\frac{1}{2} J_U(\mathbf B_h)^2+C_R J_U(\mathbf B_h)D_U(\mathbf B_h)-\tau D_U(\mathbf B_h)^2
			\\ &=-\frac12\bigl(J_U(\mathbf B_h)-C_RD_U(\mathbf B_h)\bigr)^2-\left(\tau-\frac{C_R^2}{2}\right)D_U(\mathbf B_h)^2
			\le0.
		\end{aligned}
	\end{equation}
    This completes the proof.

To establish~\eqref{eq:en_defect_bound}, the decomposition
\eqref{eq:en_source_decomposition}, orthogonality~\eqref{eq:en_source_orth},
and Cauchy--Schwarz give, for $\mathbf w,\mathbf z\in\mathbf W_h^k$,
\begin{equation}\label{eq:aim}
 \begin{aligned}
 |(\mathcal R_h(\mathbf w),\mathbf z)|
 &\le\sum_a\|\mathbf R_a(\mathbf w)\|_{\omega_a}
 \inf_{\mathbf p\in\mathbf W^k(\omega_a)}
 \|\mathbf z-\mathbf p\|_{\omega_a}\\
 &\le\left(\sum_a\|\mathbf R_a(\mathbf w)\|_{\omega_a}^2\right)^{1/2}
 \left(\sum_a\inf_{\mathbf p\in\mathbf W^k(\omega_a)}
 \|\mathbf z-\mathbf p\|_{\omega_a}^2\right)^{1/2}.
 \end{aligned}
\end{equation}
Choose an endpoint $v_e$ of each edge $e$, so that
$\omega_e\subset\omega_{v_e}$. Since only a bounded number of edges
are assigned to each vertex, Lemma~\ref{lem:en_regularity}
directly controls the second factor:
\begin{equation}\label{eq:en_second_factor}
 \sum_a\inf_{\mathbf p\in\mathbf W^k(\omega_a)}
 \|\mathbf z-\mathbf p\|_{\omega_a}^2
 \le C\sum_{v\in\mathscr V}
 \inf_{\mathbf p\in\mathbf W^k(\omega_v)}
 \|\mathbf z-\mathbf p\|_{\omega_v}^2
 \le C\sum_{v\in\mathscr V}\|\mathcal P_v\mathbf z\|_{\omega_v}^2.
\end{equation}
The following subsection estimates the first factor.

\subsection{\texorpdfstring{Bounds for the reconstruction residual}{Bounds for the reconstruction residual}}\label{subsubsec:en_constant_local}
	We estimate the local residual contributions from Lemma~\ref{lem:en_sources}. Write $u_{n,e}=\mathbf u\cdot\mathbf n_e$. For $\mathbf w\in\mathbf W_h^k$ and $e=(v_a,v_b)\in\mathscr E$, split the tangential jump $[\mathbf w\cdot\mathbf m_e]_e$ into its linear endpoint parts and a remainder vanishing at $v_a,v_b$:
	\begin{equation}\label{eq:en_jump_parts}
		\begin{aligned}
			s_{v,e}(\mathbf w)&=[\mathbf w\cdot\mathbf m_e]_e(v)\lambda_v|_e,
			&&v\in\{v_a,v_b\},\\
			s_e(\mathbf w)&=[\mathbf w\cdot\mathbf m_e]_e-s_{v_a,e}(\mathbf w)-s_{v_b,e}(\mathbf w).
		\end{aligned}
	\end{equation}
We introduce the following weighted local jump measures:
\begin{equation}\label{eq:en_local_jump_sizes}
 \eta_v(\mathbf w)^2
 =U\sum_{f\ni v}\|s_{v,f}(\mathbf w)\|_f^2,
 \qquad
 \eta_e(\mathbf w)^2=U\|s_e(\mathbf w)\|_e^2.
\end{equation}

\begin{lemma}[Numerical-flux data bounds]\label{lem:en_flux_data}
    For every $\mathbf w\in\mathbf W_h^k$, the edge flux
    \eqref{eq:scheme_LFfield} satisfies
    \begin{equation}\label{eq:en_flux_data_bounds}
        \|d_{K,f}(v)\lambda_v\|_f
        \le U\|s_{v,f}(\mathbf w)\|_f,
        \qquad
        \|r_{K,e}\|_e\le U\|s_e(\mathbf w)\|_e.
    \end{equation}
    The vertex flux~\eqref{eq:scheme_vertexE} satisfies
    \begin{equation}\label{eq:en_vertex_admissibility}
        \left|\widetilde E_v(\mathbf w)
        -\frac1{m_v}\sum_{K\ni v}E_K(v)\right|^2
        \le C\frac Uh\,\eta_v(\mathbf w)^2,
    \end{equation}
    and hence
    \begin{equation}\label{eq:en_source_scales}
        \sum_{K\ni v}|z_{K,v}|^2
        \le C\frac Uh\,\eta_v(\mathbf w)^2.
    \end{equation}
\end{lemma}
\begin{proof}
    Normal continuity gives
    $E^+-E^-=-u_{n,f}[\mathbf w\cdot\mathbf m_f]_f$.
    The LF formula~\eqref{eq:scheme_LFfield}
    and $|u_{n,f}|\le U$ give~\eqref{eq:en_flux_data_bounds}
    by the endpoint splitting~\eqref{eq:scheme_defect_split}.
    Formula~\eqref{eq:scheme_vertexE}, bounded vertex valence,
    $|s_{v,f}(v)|^2=3h_f^{-1}\|s_{v,f}\|_f^2$, and
    the definition of $\eta_v$ give~\eqref{eq:en_vertex_admissibility}
    by Cauchy--Schwarz.
    For any shared $\widetilde E_v$,
    \[
        z_{K^+,v}-z_{K^-,v}=-u_{n,f}s_{v,f}(v),
        \qquad
        \frac1{m_v}\sum_{K\ni v}z_{K,v}
        =\frac1{m_v}\sum_{K\ni v}E_K(v)-\widetilde E_v.
    \]
    Hence the discrete Poincar\'e inequality on the vertex fan and
    \eqref{eq:en_vertex_admissibility} yield
    \[
            \sum_{K\ni v}|z_{K,v}|^2\le C\sum_{f\ni v}|u_{n,f}s_{v,f}(v)|^2
            +C\left|\frac1{m_v}\sum_{K\ni v}E_K(v)-\widetilde E_v\right|^2\le C\frac Uh\,\eta_v(\mathbf w)^2,
    \]
    proving~\eqref{eq:en_source_scales}.
\end{proof}

\begin{lemma}[Stability of the local corrections]\label{lem:en_correction_stability}
    The local contributions defined in
    \eqref{eq:en_local_contribution}
    satisfy
    \begin{equation}\label{eq:en_composite_source_bounds}
 \begin{aligned}
 \|\mathbf R_v(\mathbf w)\|_{\omega_v}^2
 &\le C\left(\sum_{K\ni v}|z_{K,v}|^2
 +\sum_{K\ni v}\sum_{\substack{f\subset\partial K\\v\in f}}
 h_f^{-1}\|d_{K,f}(v)\lambda_v\|_f^2\right),\\
 \|\mathbf R_e(\mathbf w)\|_{\omega_e}^2
 &\le Ch_e^{-1}\sum_{K\supset e}\|r_{K,e}\|_e^2.
 \end{aligned}
    \end{equation}
\end{lemma}
\begin{proof}
Reference-element norm equivalence and scaling give, for
$q\in\mathbb P^{k+1}(K)$,
\[
 h_K^{-2}\|q\|_K^2+\|\nabla q\|_K^2
 \le C\left(
 h_K^{-1}\|q\|_{\partial K}^2
 +\sup_{0\ne b\in\mathcal B_K^{k+1}}
 \frac{|(\nabla q,\nabla b)_K|^2}{\|\nabla b\|_K^2}
 \right).
\]
The supremum is taken to be zero when $\mathcal B_K^{k+1}=\{0\}$. 
For $q=\Theta_{K,v}$, one-dimensional polynomial norm equivalence
bounds the boundary-trace term by its vertex values and edge moments
from~\eqref{eq:scheme_theta_vertex_trace}. The bubble equation
\eqref{eq:scheme_theta_vertex_bubble} and an inverse trace estimate
bound the supremum term. Thus
\[
\begin{aligned}
 h_K^{-2}\|\Theta_{K,v}\|_K^2
 +\|\nabla\Theta_{K,v}\|_K^2
 \le C\left(
 |z_{K,v}|^2+
 \sum_{\substack{f\subset\partial K\\v\in f}}
 h_f^{-1}\|d_{K,f}(v)\lambda_v\|_f^2
 \right).
\end{aligned}
\]
Testing~\eqref{eq:en_boundary_liftings} with $\mathbf S_v$
and applying an inverse trace estimate gives
\[
 \|\mathbf S_v\|_K^2
 \le C\sum_{\substack{f\subset\partial K\\v\in f}}
 h_f^{-1}\|d_{K,f}(v)\lambda_v\|_f^2.
\]
Moreover, the moment equation~\eqref{eq:scheme_vertex_moments},
weighted polynomial norm equivalence and an inverse estimate yield
\[
 \|\hat\nabla\times(\lambda_v\psi_v)\|_{\omega_v}
 \le Ch^{-1}\|\Theta_v\|_{\omega_v}.
\]
By~\eqref{eq:en_local_contribution}, these estimates control all
three terms in $\mathbf R_v$. The same argument applies to the edge-patch terms.  Summing the cell estimates over each
patch $\omega_a$ proves
\eqref{eq:en_composite_source_bounds}.
\end{proof}

Combining Lemmas~\ref{lem:en_flux_data} and
\ref{lem:en_correction_stability} with the endpoint inverse estimate
\[
 \|s_{v,e}(\mathbf w)\|_e^2
 \le C\|[\mathbf w\cdot\mathbf m_e]_e\|_e^2,
 \qquad v\in\{v_a,v_b\}
\]
gives the following bounds.

\begin{lemma}\label{lem:en_source_bounds}
For every $\mathbf w\in\mathbf W_h^k$,
\begin{equation}\label{eq:en_source_bound}
 \begin{aligned}
 \|\mathbf R_a(\mathbf w)\|_{\omega_a}
 &\le C\sqrt{U/h}\,\eta_a(\mathbf w),
 &&a\in\mathscr V\sqcup\mathscr E,\\
 \sum_a\eta_a(\mathbf w)^2&\le C J_U(\mathbf w)^2.
 \end{aligned}
\end{equation}
\end{lemma}

\begin{remark}[Other vertex electric fields]\label{rem:en_vertex_flux}
    Although we use the two-dimensional LF vertex flux~\eqref{eq:scheme_vertexE} from~\cite{liu2025arbitrarily}, the essential property for the analysis is~\eqref{eq:en_vertex_admissibility}. Any linear vertex flux satisfying this condition is admissible, including the simple average $\widetilde E_v=m_v^{-1}\sum_{K\ni v}E_K(v)$. Meanwhile, the dissipation in the edge fluxes $\hat E$ is required in the present analysis.
\end{remark}

By Lemma~\ref{lem:en_source_bounds}, the first factor in
\eqref{eq:aim} satisfies
\begin{equation}\label{eq:firstterm}
    \left(\sum_a
    \|\mathbf R_a(\mathbf w)\|_{\omega_a}^2
    \right)^{1/2}
    \le C\sqrt{\frac Uh}\,J_U(\mathbf w),
    \qquad \forall\mathbf w\in\mathbf W_h^k.
\end{equation}
Now we can prove the following residual estimate.
	
	\begin{lemma}\label{lem:en_defect_bound}
		There exists a constant $C_R>0$, such that
			\begin{equation}\label{eq:en_general_defect_bound}
				|(\mathcal R_h(\mathbf w),\mathbf z)|
				\le C_R J_U(\mathbf w)D_U(\mathbf z),
				\qquad
				\forall\mathbf w,\mathbf z\in\mathbf W_h^k.
			\end{equation}
	\end{lemma}
	
\begin{proof}
Lemma~\ref{lem:en_regularity} and $h_v\simeq h$ also give
\begin{equation}\label{eq:en_weighted_regularity}
 \sum_{v\in\mathscr V}\frac{U}{h_v}
 \inf_{\mathbf p\in\mathbf W^k(\omega_v)}
 \|\mathbf w-\mathbf p\|_{\omega_v}^2
 \le C D_U(\mathbf w)^2,
 \qquad \mathbf w\in\mathbf W_h^k.
\end{equation}
By~\eqref{eq:aim}, \eqref{eq:firstterm}, and the edge-to-vertex
comparison in~\eqref{eq:en_second_factor},
\[
 |(\mathcal R_h(\mathbf w),\mathbf z)|
 \le C J_U(\mathbf w)
 \left(\sum_{v\in\mathscr V}\frac{U}{h_v}
 \inf_{\mathbf p\in\mathbf W^k(\omega_v)}
 \|\mathbf z-\mathbf p\|_{\omega_v}^2\right)^{1/2}
 \le C_R J_U(\mathbf w)D_U(\mathbf z).
\]
This completes the proof.
\end{proof}

	\begin{remark}
		We have numerically studied suitable choices of $\tau$ on several well-shaped meshes. For $0\le k\le 4$, we found the following penalty values suitable in these tests:
		\begin{equation}\label{eq:tau}
			(\tau_0,\tau_1,\tau_2,\tau_3,\tau_4)=(0.2,\  0.6,\ 4.6,\ 9,\ 15).
		\end{equation}
		For lower-quality meshes, these parameters may need to be re-estimated.
	\end{remark}

\section{Error estimates}\label{sec:error}
In this section, we establish an a priori error estimate for the
proposed scheme. Assume that the exact solution of
\eqref{eq:intro_induction} is periodic and divergence-free, with $\mathbf B\in C^1([0,T];[H^{k+1}(\Omega)]^2)$. The definition of $G_h$ extends to
$\prod_{K\in\mathcal T_h}[H^1(K)]^2$.
For the exact solution $\mathbf B$, consistency gives
\[
 (\partial_t\mathbf B,\mathbf z)
 =(G_h(\mathbf B),\mathbf z),
 \qquad \forall\mathbf z\in\mathbf V_h^k.
\]

\subsection{Main result}\label{subsec:err_result}

\begin{theorem}[Error estimate]\label{thm:err_main}
		Fix $k\ge0$, $T>0$, and $\delta>0$ independently of $h$. Let $C_R$ be the constant in Lemma~\ref{lem:en_defect_bound}, and suppose
		\begin{equation}\label{eq:err_tau}
			\tau\ge\frac{C_R^2}{2}+\delta.
		\end{equation}
		 Then for any initial approximation $\mathbf B_h(0)\in\mathbf W_h^k$, the semidiscrete solution $\mathbf B_h(t)$ of~\eqref{eq:scheme_final_rate} satisfies
		\begin{equation}\label{eq:err_main_bound}
			\begin{aligned}
				\sup_{0\le t\le T}\|\mathbf B(t)-\mathbf B_h(t)\|
				\le{}&\|\mathbf B(0)-\mathbf B_h(0)\|
				+Ch^{k+1}\|\mathbf B\|_{L^\infty(0,T;H^{k+1})}\\
				&+Ch^{k+1}\|\mathbf B_t\|_{L^1(0,T;H^{k+1})}+C\sqrt U\,h^{k+1/2}\|\mathbf B\|_{L^2(0,T;H^{k+1})}.
			\end{aligned}
		\end{equation}
		In particular, if $\mathbf B_0=\mathbf B(0)\in[H^{k+2}(\Omega)]^2$ and
		\begin{equation}\label{eq:err_initial_accuracy}
			\|\mathbf B_0-\mathbf B_h(0)\|
			\le C_{\rm init}h^{k+1/2}\|\mathbf B_0\|_{H^{k+1}},
		\end{equation}
		then for fixed $T$ and $\mathbf u$,
		\begin{equation}\label{eq:err_rate}
			\sup_{0\le t\le T}\|\mathbf B(t)-\mathbf B_h(t)\|
			\le C_T h^{k+1/2}\|\mathbf B_0\|_{H^{k+2}},
		\end{equation}
		where $C$ depends only on $k$, $\Omega$, the mesh regularity bounds, $\tau$, and $\delta$, and $C_T$ may additionally depend on $C_{\rm init}$, $T$, and $U$.
	\end{theorem}

\subsection{A compatible local projection}\label{subsec:err_projection}

Following the classical DG error estimates, we use the error splitting
\begin{equation}\label{eq:err_split}
		\mathbf e_\pi=\mathbf B-\Pi_h\mathbf B,\qquad
		\mathbf e_h=\mathbf B_h-\Pi_h\mathbf B,\qquad
		\mathbf B-\mathbf B_h=\mathbf e_\pi-\mathbf e_h,
	\end{equation}
Define the projection on the elementwise divergence-free space
\begin{equation}\label{eq:err_broken_space}
			X_h^{k+1}
			=\left\{
			\mathbf B\in
			\prod_{K\in\mathcal T_h}[H^{k+1}(K)]^2:
			\nabla\cdot(\mathbf B|_K)=0
			\right\}.
		\end{equation}
It contains $\mathbf W_h^k$ and
$[H^{k+1}(\Omega)]^2\cap\ker(\nabla\cdot)$.
On each $K$, choose the zero-mean stream function
$\phi_K\in H^{k+2}(K)\hookrightarrow C^0(\overline K)$ satisfying~\cite{costabel2010bogovskii,taylor2011partial}
\begin{equation}\label{eq:err_stream_function}
			\hat\nabla\times\phi_K=\mathbf B|_K,
			\qquad
			|\phi_K|_{H^{m+1}(K)}
			\le C|\mathbf B|_{H^m(K)},
			\quad 1\le m\le k+1.
		\end{equation}
Let $I_K:H^{k+2}(K)\to\mathbb P^{k+1}(K)$ be the
$H^1$-conforming moment interpolant
\begin{equation}\label{eq:err_potential_interpolant}
		\begin{aligned}
			I_K\phi_K(v)&=\phi_K(v),
			&&\forall v\text{ a vertex of }K,\\
			(I_K\phi_K-\phi_K,q)_e&=0,
			&&\forall q\in\mathbb P^{k-1}(e),\quad
			\forall e\subset\partial K,\\
			(I_K\phi_K-\phi_K,q)_K&=0,
			&&\forall q\in\mathbb P^{k-2}(K).
		\end{aligned}
	\end{equation}
We define
\begin{equation}\label{eq:err_projection}
		(\Pi_h\mathbf B)|_K=\hat\nabla\times I_K\phi_K,
		\qquad \forall K\in\mathcal T_h.
	\end{equation}
Because $I_K(\phi_K+c)=I_K\phi_K+c$, this definition is independent
of the additive constant in $\phi_K$.
Normal continuity of $\mathbf B$ makes adjacent stream-function
traces differ by a constant; interpolation preserves that difference,
so $\Pi_h\mathbf B\in\mathbf W_h^k$.
The shared vertex values and edge moments define the conforming interpolant
$I_h:H^{k+2}(\Omega)\to S_h^{k+1}$ by
$(I_h\phi)|_K=I_K(\phi|_K)$, with the commuting diagram
\begin{equation}\label{eq:err_commuting_diagram}
		\begin{array}{ccc}
			H^{k+2}(\Omega)
			&\xrightarrow{\ \hat\nabla\times\ }&
			[H^{k+1}(\Omega)]^2\cap\ker(\nabla\cdot)
			\\[1mm]
			{\scriptstyle I_h}\big\downarrow
			&&
			\big\downarrow{\scriptstyle\Pi_h}
			\\[1mm]
			S_h^{k+1}
			&\xrightarrow{\ \hat\nabla\times\ }&
			\mathbf W_h^k.
		\end{array}
	\end{equation}

\begin{lemma}[Projection and approximation properties]\label{lem:err_projection}
		For every $\mathbf B\in X_h^{k+1}$ with continuous normal traces across $\mathscr E$, the error $\mathbf e_\pi=\mathbf B-\Pi_h\mathbf B$ satisfies the moment orthogonality conditions
		\begin{align}
			(\mathbf e_\pi,\mathbf q)_K&=0,
			&&\forall\mathbf q\in[\mathbb P^{k-1}(K)]^2,
			\quad\forall K\in\mathcal T_h,
			\label{eq:err_volume_moments}\\
			(\mathbf e_\pi\cdot\mathbf n_e,q)_e&=0,
			&&\forall q\in\mathbb P^k(e),
			\quad\forall e\in\mathscr E,
			\label{eq:err_normal_moments}
		\end{align}
		and the approximation bounds for all $K\in\mathcal T_h$:
		\begin{align}
			\|\mathbf e_\pi\|_K
			+h_K|\mathbf e_\pi|_{H^1(K)}
			&\le Ch_K^{k+1}|\mathbf B|_{H^{k+1}(K)},
			\label{eq:err_approximation}\\
			\|\mathbf e_\pi\|_{\partial K}^2
			&\le Ch_K^{2k+1}|\mathbf B|_{H^{k+1}(K)}^2.
			\label{eq:err_trace_approximation}
		\end{align}
		On a fixed mesh, every $\mathbf B\in C^1([0,T];X_h^{k+1})$ also satisfies
			\begin{equation}\label{eq:err_time_commutation}
				\partial_t(\Pi_h\mathbf B)=\Pi_h(\partial_t\mathbf B).
			\end{equation}
	\end{lemma}

\begin{proof}
With $\epsilon_K=\phi_K-I_K\phi_K$, Green's formula with $\mathbf q\in[\mathbb P^{k-1}(K)]^2$ and integration by parts along $e=(v_a,v_b)$ against $q\in\mathbb P^k(e)$ give
		\[
		\begin{aligned}
			(\mathbf e_\pi,\mathbf q)_K&=(\epsilon_K,\hat\nabla\times\mathbf q)_K-\sum_{e\subset\partial K}(\epsilon_K,\mathbf q\cdot\mathbf m_K)_e=0,\\
			(\mathbf e_\pi\cdot\mathbf n_e,q)_e&=[\epsilon_K q]_{v_a}^{v_b}-(\epsilon_K,\partial_{\mathbf m_e}q)_e=0,
		\end{aligned}
		\]
		by~\eqref{eq:err_potential_interpolant}, proving~\eqref{eq:err_volume_moments}--\eqref{eq:err_normal_moments}. 

Bramble--Hilbert scaling and the trace inequality for $I_K$~\cite{ern2021finite},
together with~\eqref{eq:err_stream_function}, give
\eqref{eq:err_approximation}--\eqref{eq:err_trace_approximation}.
Boundedness and linearity on a fixed mesh give~\eqref{eq:err_time_commutation}
for $C^1$ curves~\cite{hunter2001applied}.
\end{proof}

\subsection{Consistency bounds}\label{subsec:err_consistency}
The following lemma bounds the DG residual and jumps of the projection error.

\begin{lemma}\label{lem:err_DG}
		Let $\mathbf B\in[H^{k+1}(\Omega)]^2$ be divergence-free and set $\mathbf e_\pi=\mathbf B-\Pi_h\mathbf B$. For every $\mathbf z\in\mathbf W_h^k$, the DG residual satisfies
		\begin{equation}\label{eq:err_reduced_flux}
			(G_h(\mathbf e_\pi),\mathbf z)
			=\sum_{e\in\mathscr E}\left(
			u_{n,e}\{\mathbf e_\pi\cdot\mathbf m_e\}_e
			-\frac{U}{2}[\mathbf e_\pi\cdot\mathbf m_e]_e,
			[\mathbf z\cdot\mathbf m_e]_e
			\right)_e,
		\end{equation}
		where $u_{n,e}=\mathbf u\cdot\mathbf n_e$, and
		\begin{align}
			|(G_h(\mathbf e_\pi),\mathbf z)|
			&\le C\sqrt U\,h^{k+1/2}|\mathbf B|_{H^{k+1}(\Omega)} J_U(\mathbf z),
			\label{eq:err_DG_residual}\\
			J_U(\Pi_h\mathbf B)^2
			&\le CUh^{2k+1}|\mathbf B|_{H^{k+1}(\Omega)}^2.
			\label{eq:err_jump_consistency}
		\end{align}
	\end{lemma}

\begin{proof}
Since $\mathbf u$ is constant, \eqref{eq:err_volume_moments}
eliminates the volume term in~\eqref{eq:DG}. Pairing the cell
contributions on each edge gives
\[
 (G_h(\mathbf e_\pi),\mathbf z)
 =\sum_{e\in\mathscr E}\Bigl(
 u_{n,e}\{\mathbf e_\pi\cdot\mathbf m_e\}_e
 -(\mathbf u\cdot\mathbf m_e)\{\mathbf e_\pi\cdot\mathbf n_e\}_e-\frac{U}{2}
   [\mathbf e_\pi\cdot\mathbf m_e]_e,\,
 [\mathbf z\cdot\mathbf m_e]_e
 \Bigr)_e.
\]
Since
$(\mathbf u\cdot\mathbf m_e)[\mathbf z\cdot\mathbf m_e]_e
\in\mathbb P^k(e)$, the normal-component term vanishes by
\eqref{eq:err_normal_moments}, proving~\eqref{eq:err_reduced_flux}.

The trace estimate~\eqref{eq:err_trace_approximation} gives
\[
 \sum_{e\in\mathscr E}\left(
 \|\{\mathbf e_\pi\cdot\mathbf m_e\}_e\|_e^2
 +\|[\mathbf e_\pi\cdot\mathbf m_e]_e\|_e^2
 \right)
 \le Ch^{2k+1}|\mathbf B|_{H^{k+1}(\Omega)}^2.
\]
Together with $|u_{n,e}|\le U$, this bound and
Cauchy--Schwarz yield~\eqref{eq:err_DG_residual}. Since the traces of $\mathbf B$ agree across each edge,
$[\Pi_h\mathbf B\cdot\mathbf m_e]_e
=-[\mathbf e_\pi\cdot\mathbf m_e]_e$;
the same bound therefore gives~\eqref{eq:err_jump_consistency}.
\end{proof}

\begin{lemma}\label{lem:err_patch_consistency}
		For every divergence-free field $\mathbf B\in[H^{k+1}(\Omega)]^2$,
		\begin{equation}\label{eq:err_patch_consistency}
			D_U(\Pi_h\mathbf B)^2
			\le CUh^{2k+1}|\mathbf B|_{H^{k+1}(\Omega)}^2.
		\end{equation}
	\end{lemma}

\begin{proof}
    On each ball $\mathcal O_v=B_{2h}(v)\supset\omega_v$,
    Bramble--Hilbert approximation~\cite{bramble1970estimation,dupont1980polynomial}
    gives a common polynomial $\mathbf p_v\in[\mathbb P^k(\mathcal O_v)]^2$ with
    $\|\mathbf B-\mathbf p_v\|_{\omega_v}
    \le Ch^{k+1}|\mathbf B|_{H^{k+1}(\mathcal O_v)}$.
    Polynomial cancellation~\eqref{eq:scheme_polynomial_consistency}
    and~\eqref{eq:scheme_zeta}, tested with $\zeta_v(\Pi_h\mathbf B)$, give
    \begin{equation}\label{eq:err_patch_projection_estimate}
        \|\mathcal P_v\Pi_h\mathbf B\|_{\omega_v}
        \le Ch^{k+1}|\mathbf B|_{H^{k+1}(\mathcal O_v)}
        +\|\mathbf e_\pi\|_{\omega_v}.
    \end{equation}
    Square~\eqref{eq:err_patch_projection_estimate} and sum with weights
    $U/h_v$. By shape regularity and quasi-uniformity, the
    neighborhoods $\mathcal O_v$ have uniformly bounded overlap, including
    across periodic boundaries. Together with~\eqref{eq:err_approximation},
    this gives~\eqref{eq:err_patch_consistency}.
\end{proof}

\subsection{The proof of Theorem~\texorpdfstring{\ref{thm:err_main}}{4.1}}\label{subsec:err_proof}

\begin{proof}

By~\eqref{eq:en_decomposition} and the $L^2$-orthogonality of $\mathcal P_v$, subtracting the exact consistency relation $(\mathbf B_t,\mathbf z)=(G_h(\mathbf B),\mathbf z)$ from~\eqref{eq:scheme_final_rate} yields the error equation
\begin{equation}\label{eq:err_equation}
			\begin{aligned}
				(\partial_t\mathbf e_h,\mathbf z)
				={}&(\mathcal L_h(\mathbf e_h),\mathbf z)
				-(G_h(\mathbf e_\pi),\mathbf z)+(\partial_t\mathbf e_\pi,\mathbf z)\\
				&+(\mathcal R_h(\Pi_h\mathbf B),\mathbf z)-\tau\sum_{v\in\mathscr V}\frac U{h_v}(\mathcal P_v\Pi_h\mathbf B,\mathcal P_v\mathbf z)_{\omega_v},
			\end{aligned}
		\end{equation}
for all $\mathbf z\in\mathbf W_h^k$.
Under~\eqref{eq:err_tau}, completing the square as in
\eqref{eq:en_constant_energy} gives, for some $c>0$ independent of $h$ and every $\mathbf w\in\mathbf W_h^k$,
\begin{equation}\label{eq:err_coercivity}
    \frac12\bigl(J_U(\mathbf w)-C_RD_U(\mathbf w)\bigr)^2
    +\delta D_U(\mathbf w)^2
    \ge c\bigl(J_U(\mathbf w)^2+D_U(\mathbf w)^2\bigr).
\end{equation}
Testing~\eqref{eq:err_equation} with $\mathbf z=\mathbf e_h$ and using the DG energy identity~\eqref{eq:energy_dg},
Lemmas~\ref{lem:en_defect_bound}, \ref{lem:err_DG},
and~\ref{lem:err_patch_consistency}, together with
Cauchy--Schwarz, give
\begin{equation}\label{eq:err_preabsorption}
    \begin{aligned}
        \frac12\frac{\mathrm d}{\mathrm dt}\|\mathbf e_h\|^2
        &+c\bigl(J_U(\mathbf e_h)^2+D_U(\mathbf e_h)^2\bigr)
        \\\le{}&C\sqrt U\,h^{k+1/2}|\mathbf B(t)|_{H^{k+1}}
        \bigl[J_U(\mathbf e_h)+D_U(\mathbf e_h)\bigr]
        +\|\partial_t\mathbf e_\pi\|\,\|\mathbf e_h\|.
    \end{aligned}
\end{equation}
Young's inequality absorbs the terms involving $J_U(\mathbf e_h)$ and
$D_U(\mathbf e_h)$; discarding the remaining nonnegative dissipation yields
\begin{equation}\label{eq:err_energy}
        \frac12\frac{\mathrm d}{\mathrm dt}\|\mathbf e_h\|^2
        \le{}CUh^{2k+1}|\mathbf B(t)|_{H^{k+1}}^2+\|\partial_t\mathbf e_\pi\|\,\|\mathbf e_h\|.
\end{equation}
A standard energy argument applied to~\eqref{eq:err_energy},
together with the projection properties
\eqref{eq:err_approximation} and~\eqref{eq:err_time_commutation},
yields~\eqref{eq:err_main_bound}. 
Finally, the periodic translation
$\mathbf B(\mathbf x,t)=\mathbf B_0(\mathbf x-\mathbf u t)$
yields the bounds~\cite{taylor2011partial}
\begin{equation}\label{eq:classic_transport}
    \|\mathbf B(t)\|_{H^{k+1}(\Omega)}
    =\|\mathbf B_0\|_{H^{k+1}(\Omega)},\qquad
    \|\partial_t\mathbf B(t)\|_{H^{k+1}(\Omega)}
    \le CU\|\mathbf B_0\|_{H^{k+2}(\Omega)}.
\end{equation}
Together with~\eqref{eq:err_initial_accuracy}, these bounds
give~\eqref{eq:err_rate}.
\end{proof}

\section{Numerical experiments}\label{sec:numerical}
We test~\eqref{eq:scheme_final_rate} for $0\le k\le 4$, with the penalties in~\eqref{eq:tau}, on periodic unstructured triangular meshes of $\Omega=[0,1)^2$. Following~\cite{yang2016stability}, we take $\mathbf u=(1,1)^{\mathsf T}$ and
\begin{equation}\label{eq:num_exact}
  B_x(x,y,t)=-\sin\bigl(2\pi(y-t)\bigr),\qquad
  B_y(x,y,t)=\sin\bigl(2\pi(x-t)\bigr).
\end{equation}
The exact solution has period $1$ and magnetic energy $\frac12\|\mathbf B(t)\|^2=\frac12$. We generate periodic Delaunay meshes with target sizes $1/N$, $N=8,16,32,64$, and effective size $h_{\rm eff}=(2/N_K)^{1/2}$, where $N_K$ is the cell count. Here $U=\sqrt2$, and we integrate to $T=1$ using the ten-stage, fourth-order SSP-RK method~\cite{ketcheson2008highly}, with $\Delta t=h_{\rm eff}/(2k+1)$.

Figure~\ref{fig:num_convergence} reports the final-time $L^2$ errors. The convergence rates on the finest pair of meshes are $0.90$, $2.09$, $3.13$, $4.03$, and $5.18$ for $0\le k\le 4$, respectively, consistent with order $k+1$ on these test meshes.

\begin{figure}
  \centering
  \includegraphics[width=.82\textwidth]{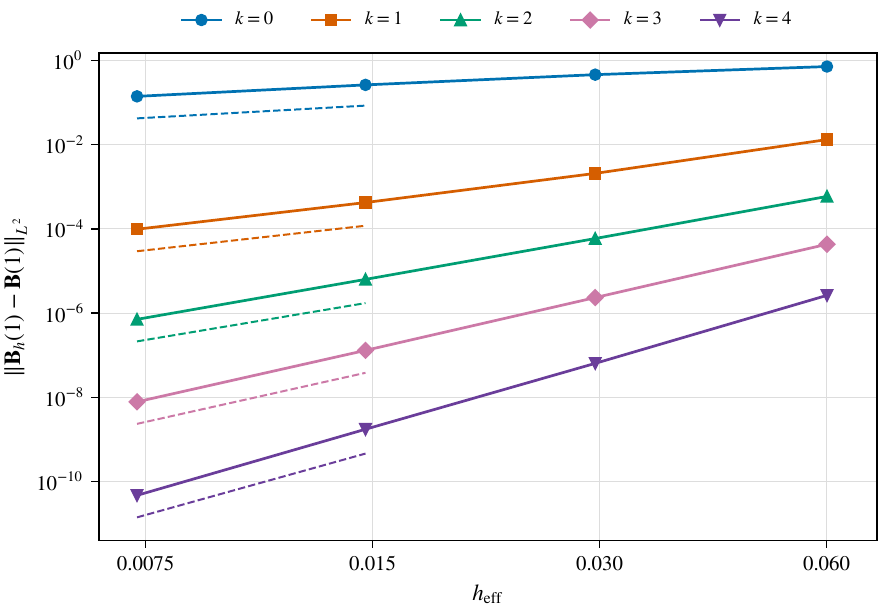}
  \caption{Final-time $L^2$ errors on four periodic unstructured meshes for $0\le k\le 4$. Dashed reference lines have slopes $k+1$.}
  \label{fig:num_convergence}
\end{figure}

On the $N=16$ mesh, Table~\ref{tab:num_energy} records the energy $\mathcal E_h(t)=\frac12\|\mathbf B_h(t)\|^2$, its semidiscrete rate $\dot{\mathcal E}_h(t)=(\mathbf B_h(t),\mathcal L_h(\mathbf B_h(t)))$, and the single-step increments over $[0,1]$. The recorded rates and increments are negative for every degree. The rates are consistent with Theorem~\ref{thm:en_constant}; the negative increments are a numerical observation for the chosen time integrator. Across all five degrees, the maximum cellwise divergence is $4.53\times10^{-12}$ and the maximum normal-jump basis coefficient is $6.31\times10^{-13}$, consistent with GDF preservation to roundoff.

\begin{table}[htbp]
		\centering\small
		\caption{Energy diagnostics over one period on the mesh with $N=16$.
			The maxima are taken over the recorded output times for each degree.}
		\label{tab:num_energy}
		\begin{tabular}{r r@{$\,\times\,$}l r@{$\,\times\,$}l r@{$\,\times\,$}l}
			\hline
			$k$
			& \multicolumn{2}{c}{$\mathcal E_h(1)-\mathcal E_h(0)$}
			& \multicolumn{2}{c}{$\max\dot{\mathcal E}_h$}
			& \multicolumn{2}{c}{$\max\Delta\mathcal E_h$}\\
			\hline
			0 & $-3.513$ & $10^{-1}$  & $-1.807$ & $10^{-1}$  & $-1.040$ & $10^{-3}$\\
			1 & $-1.276$ & $10^{-3}$  & $-8.585$ & $10^{-4}$  & $-3.875$ & $10^{-6}$\\
			2 & $-6.627$ & $10^{-6}$  & $-6.380$ & $10^{-6}$  & $-1.868$ & $10^{-8}$\\
			3 & $-4.516$ & $10^{-9}$  & $-4.336$ & $10^{-9}$  & $-1.810$ & $10^{-11}$\\
			4 & $-3.162$ & $10^{-12}$ & $-2.656$ & $10^{-12}$ & $-2.115$ & $10^{-14}$\\
			\hline
		\end{tabular}
	\end{table}

\section{Concluding remarks}\label{sec7}

We have established semidiscrete energy stability and an $L^2$ error bound of order $k+1/2$ for the proposed GDF-DG method on periodic triangular meshes. Numerical experiments preserve global divergence freedom to roundoff and show energy dissipation and order $k+1$ convergence for a smooth test problem. Extending the analysis to three dimensions and applying the proposed method to the full MHD system are directions for future work.

\appendix
\section{Proof of the patch regularity}
\label{app:patch_regularity} For the next three lemmas, fix $\mathbf w\in\mathbf W_h^k$,
take its zero-mean potential $\phi_h$ from~\eqref{eq:compatible_potential},
and set
\[
\zeta_v=\zeta_v(\mathbf w),\qquad
 a(\phi,\eta)=(\nabla\phi,\nabla\eta).
\]

\begin{lemma}
\label{lem:en_local_multipliers}
For every $v\in\mathscr V$, there exist $q_v\in\mathbb P^{k-1}(\omega_v)$ such that
\begin{equation}\label{eq:en_local_multiplier}
 a(\phi_h,\eta)=(q_v,\eta)_{\omega_v}+a(\zeta_v,\eta),
 \qquad \eta\in S_0^{k+1}(\omega_v).
\end{equation}
For every edge $e=(v,v')$, these multipliers satisfy
\begin{equation}\label{eq:en_multiplier_difference}
 \|q_v-q_{v'}\|_{\omega_e}
 \le Ch^{-1}\bigl(
 \|\mathcal P_v\mathbf w\|_{\omega_v}
 +\|\mathcal P_{v'}\mathbf w\|_{\omega_{v'}}\bigr).
\end{equation}
Moreover, there exists a smooth periodic function $f_0\in C^\infty(\Omega)$ with
\begin{equation}\label{eq:en_smooth_patch_budget}
 \sum_v\|q_v-f_0\|_{\omega_v}^2+h^{2k}|f_0|_{H^k}^2
 \le C\sum_{e=(v,v')}\|q_v-q_{v'}\|_{\omega_e}^2
 \le Ch^{-2}\sum_v\|\mathcal P_v\mathbf w\|_{\omega_v}^2.
\end{equation}
\end{lemma}
\begin{proof}
Since $a(\phi_h-\zeta_v,\cdot)$ vanishes on $H_v$
by~\eqref{eq:compatible_patch_projection}, finite-dimensional
duality gives~\eqref{eq:en_local_multiplier}.
Subtract the identities associated with the endpoints of
$e=(v,v')$ and test with
$\lambda_e(q_v-q_{v'})$, which belongs to both local spaces
and has degree at most $k+1$.
Weighted norm equivalence and the inverse inequality give
\eqref{eq:en_multiplier_difference}.

To combine the local multipliers $q_v$, we construct a smooth periodic
function $f_0$. Choose a nonnegative smooth function $\rho$ supported in the unit ball $B_1(0)$ with $\int_{\mathbb R^2}\rho\,\mathrm d\mathbf x=1$. For a
sufficiently small fixed $\varepsilon>0$, set
\begin{equation}\label{eq:app_smoothing_weights}
 \rho_{\varepsilon h}=(\varepsilon h)^{-2}
 \rho\bigl(\,\cdot\,/(\varepsilon h)\bigr),\qquad
 g_v=\rho_{\varepsilon h}*\mathbf1_{\{\lambda_v>1/6\}},\qquad
 \chi_v=\frac{g_v}{\sum_{v'}g_{v'}}.
\end{equation}
For sufficiently small \(\varepsilon\), we have
\(\sum_{v'}g_{v'}(\mathbf x)\ge1\), so \(\{\chi_v\}\) is a smooth nonnegative partition of unity satisfying
\begin{equation}\label{eq:app_partition_props}
 \operatorname{supp}\chi_v\subset\{\lambda_v>0\},\qquad
 \|\partial^\alpha\chi_v\|_\infty\le C_\alpha h^{-|\alpha|},
 \quad \alpha\in\mathbb N_0^2,
\end{equation}
and set $f_0=\sum_v\chi_vq_v$. This construction is
analogous to the partition-of-unity approach
in~\cite{babuska1997partition}.

For a cell $K$ with vertex set $V(K)$, fix $v\in V(K)$. Since \(\sum_{v'\in V(K)}\chi_{v'}=1\), we obtain
\begin{equation}\label{eq:app_patch_identity}
 f_0-q_v=\sum_{v'\in V(K)}\chi_{v'}(q_{v'}-q_v)\qquad\text{on }K.
\end{equation}
Using~\eqref{eq:app_patch_identity}, \eqref{eq:app_partition_props},
$\partial^\alpha q_v=0$ for $|\alpha|=k$, the Leibniz rule, and
polynomial inverse estimates yields
\begin{equation}\label{eq:app_patch_local_bound}
 \|f_0-q_v\|_K+h^k|f_0|_{H^k(K)}
 \le C\sum_{v'\in V(K)}\|q_{v'}-q_v\|_K.
\end{equation}
Summing~\eqref{eq:app_patch_local_bound} over cells and vertices and
applying~\eqref{eq:en_multiplier_difference} gives
\eqref{eq:en_smooth_patch_budget}.
\end{proof}

\begin{lemma}[Reconstruction of mean-zero loads]
\label{lem:en_load_reconstruction}
There exist periodic functions $f\in H^k(\Omega)$ and
$r\in L^2(\Omega)$ such that
\begin{equation}\label{eq:en_reconstructed_load_budget}
 \begin{aligned}
 a(\phi_h,\eta_h)&=(f+r,\eta_h),
 &&\forall\eta_h\in S_h^{k+1},\\
 \int_\Omega f\,\mathrm d\Omega&=\int_\Omega r\,\mathrm d\Omega=0,\\
 h^{2k+2}|f|_{H^k}^2+h^2\|r\|^2&\le C\sum_v\|\mathcal P_v\mathbf w\|_{\omega_v}^2.
 \end{aligned}
\end{equation}
\end{lemma}
\begin{proof}
We decompose \(\eta_h=\sum_v\eta_v\), with \(\eta_v\in S_0^{k+1}(\omega_v)\), by distributing each Lagrange nodal coefficient equally among the vertices of its associated mesh entity.
Take $q_v,f_0$ from Lemma~\ref{lem:en_local_multipliers}. Reference-element norm equivalence,
scaling, and the inverse inequality~\cite{ern2021finite} yield
\begin{equation}\label{eq:en_nodal_splitting}
 \sum_v\bigl(\|\eta_v\|_{\omega_v}^2
 +h^2\|\nabla\eta_v\|_{\omega_v}^2\bigr)
 \le C\|\eta_h\|^2.
\end{equation}
Summing~\eqref{eq:en_local_multiplier} and applying
\eqref{eq:en_smooth_patch_budget} and~\eqref{eq:en_nodal_splitting}
gives
\begin{equation}\label{eq:en_global_residual_estimate}
 \begin{aligned}
 |a(\phi_h,\eta_h)-(f_0,\eta_h)|
 &=\left|\sum_v a(\zeta_v,\eta_v)
 +\sum_v(q_v-f_0,\eta_v)_{\omega_v}\right|\\
 &\le Ch^{-1}\left(\sum_v\|\mathcal P_v\mathbf w\|_{\omega_v}^2\right)^{1/2}\|\eta_h\|.
 \end{aligned}
\end{equation}
Let $r_h\in S_h^{k+1}$ be the $L^2$ Riesz representative of the residual
$\eta_h\mapsto a(\phi_h,\eta_h)-(f_0,\eta_h)$, so that
\begin{equation}\label{eq:app_riesz_representation}
 (r_h,\eta_h)=a(\phi_h,\eta_h)-(f_0,\eta_h),
 \qquad \forall\eta_h\in S_h^{k+1}.
\end{equation}
Taking $\eta_h=r_h$ and $\eta_h=1$ in
\eqref{eq:app_riesz_representation} and using
\eqref{eq:en_global_residual_estimate} yield
\begin{equation}\label{eq:en_residual_mean_bounds}
 \begin{aligned}
 \|r_h\|&\le Ch^{-1}\left(\sum_v\|\mathcal P_v\mathbf w\|_{\omega_v}^2\right)^{1/2},\\
 |\Omega|^{1/2}|m|&\le\|r_h\|,\qquad
 m=|\Omega|^{-1}\int_\Omega f_0\,\mathrm d\Omega.
 \end{aligned}
\end{equation}
Since $\int_\Omega r_h\,\mathrm d\Omega=-\int_\Omega f_0\,\mathrm d\Omega$, the functions
$f=f_0-m$ and $r=r_h+m$ have zero mean.
The mean-zero projection is nonexpansive in the $H^k$ seminorm, and
\eqref{eq:en_residual_mean_bounds} gives
\begin{equation}\label{eq:app_mean_projection_bounds}
 |f|_{H^k}\le |f_0|_{H^k},\qquad
 \|r\|\le \|r_h\|+|\Omega|^{1/2}|m|\le 2\|r_h\|.
\end{equation}
Together with~\eqref{eq:en_smooth_patch_budget} and
\eqref{eq:en_residual_mean_bounds}, \eqref{eq:app_mean_projection_bounds} yields
\eqref{eq:en_reconstructed_load_budget}.
\end{proof}

\begin{lemma}[Poisson--Ritz patch approximation]
\label{lem:en_poisson_approximation}
Let $f\in H^k(\Omega)$ and $r\in L^2(\Omega)$ be periodic,
zero-mean functions satisfying
$a(\phi_h,\eta_h)=(f+r,\eta_h)$ for every
$\eta_h\in S_h^{k+1}$.
If $\Phi$ is the zero-mean periodic solution of $-\Delta\Phi=f+r$,
then
\begin{equation}\label{eq:en_patch_approximation}
 \begin{aligned}
 \sum_v\inf_{\mathbf p\in\mathbf W^k(\omega_v)}
 \|\mathbf w-\mathbf p\|_{\omega_v}^2
 &\le C\left(
 |\phi_h-\Phi|_{H^1}^2
 +\sum_v\inf_{\pi\in\mathbb P^{k+1}(\mathcal O_v)}
 |\Phi-\pi|_{H^1(\omega_v)}^2\right)\\
 &\le C\bigl(h^{2k+2}|f|_{H^k}^2+h^2\|r\|^2\bigr),
 \end{aligned}
\end{equation}
where $\mathcal O_v=B_{2h}(v)\supset\omega_v$.
\end{lemma}
\begin{proof}
Write $\Phi=\Phi_f+\Phi_r$, where $\Phi_f$ and $\Phi_r$ are the zero-mean periodic Poisson solutions of $-\Delta\Phi_f=f$ and $-\Delta\Phi_r=r$, respectively. Periodic elliptic regularity and the zero-mean Poincar\'e inequality~\cite[Chaps.~3 and~5]{taylor2011partial} give
\begin{equation}\label{eq:app_poisson_regularity}
 |\Phi_f|_{H^{k+2}}\le C|f|_{H^k},\qquad
 |\Phi_r|_{H^2}\le C\|r\|.
\end{equation}
By the variational identity in the lemma and $-\Delta\Phi=f+r$,
$\phi_h$ is the Ritz projection of $\Phi$. Galerkin orthogonality therefore gives
$|\phi_h-\Phi|_{H^1}\le
\inf_{\eta_h\in S_h^{k+1}}|\Phi-\eta_h|_{H^1}$.
Conforming approximation~\cite[Chap.~22]{ern2021finite} and~\eqref{eq:app_poisson_regularity} yield
\begin{equation}\label{eq:app_ritz_bound}
|\phi_h-\Phi|_{H^1}^2\le C\bigl(
h^{2k+2}|\Phi_f|_{H^{k+2}}^2
+h^2|\Phi_r|_{H^2}^2
\bigr)\le C\bigl(
h^{2k+2}|f|_{H^k}^2+h^2\|r\|^2
\bigr).
\end{equation}
Local Bramble--Hilbert approximation~\cite{dupont1980polynomial}
gives $\pi_{f,v}\in\mathbb P^{k+1}(\mathcal O_v)$ and
$\pi_{r,v}\in\mathbb P^1(\mathcal O_v)$ such that
\begin{equation}\label{eq:app_local_bh}
\begin{aligned}
|\Phi_f-\pi_{f,v}|_{H^1(\mathcal O_v)}
\le Ch^{k+1}|\Phi_f|_{H^{k+2}(\mathcal O_v)},\quad 
|\Phi_r-\pi_{r,v}|_{H^1(\mathcal O_v)}\le Ch|\Phi_r|_{H^2(\mathcal O_v)}.
\end{aligned}
\end{equation}
Set $\pi_v=\pi_{f,v}+\pi_{r,v}$.
For any $\pi\in\mathbb P^{k+1}(\mathcal O_v)$, take
$\mathbf p_v=\overline{\mathbf w}+\hat\nabla\times\pi\in\mathbf W^k(\omega_v)$.
By~\eqref{eq:compatible_potential},
\begin{equation}\label{eq:app_vector_bound}
\|\mathbf w-\mathbf p_v\|_{\omega_v}
\le |\phi_h-\Phi|_{H^1(\omega_v)}
 +|\Phi-\pi|_{H^1(\omega_v)}.
\end{equation}
Taking infima in~\eqref{eq:app_vector_bound} and summing over $v$, with bounded patch overlap, gives the first
inequality in~\eqref{eq:en_patch_approximation}. Choosing $\pi=\pi_v$, \eqref{eq:app_local_bh} and
$\omega_v\subset\mathcal O_v$ give
\begin{equation}\label{eq:app_local_sum}
\begin{aligned}
\sum_v\inf_{\pi\in\mathbb P^{k+1}(\mathcal O_v)}
|\Phi-\pi|_{H^1(\omega_v)}^2\le{}&C\sum_v\left(
h^{2k+2}|\Phi_f|_{H^{k+2}(\mathcal O_v)}^2
+h^2|\Phi_r|_{H^2(\mathcal O_v)}^2
\right)\\
\le{}&C\left(
h^{2k+2}|\Phi_f|_{H^{k+2}}^2
+h^2|\Phi_r|_{H^2}^2
\right)\\
\le{}&C\left(
h^{2k+2}|f|_{H^k}^2+h^2\|r\|^2
\right).
\end{aligned}
\end{equation}
Together with~\eqref{eq:app_ritz_bound},
\eqref{eq:app_local_sum} yields~\eqref{eq:en_patch_approximation}.
\end{proof}

\begin{proof}[Proof of Lemma~\ref{lem:en_regularity}]
Polynomial consistency~\eqref{eq:scheme_polynomial_consistency}
and projection orthogonality give $
\|\mathcal P_v\mathbf w\|_{\omega_v}
\le\|\mathbf w-\mathbf p\|_{\omega_v},
\ \mathbf p\in\mathbf W^k(\omega_v)$. Taking the infimum and summing proves the first inequality. The reverse bound follows by combining
\eqref{eq:en_reconstructed_load_budget} and
\eqref{eq:en_patch_approximation}.
\end{proof}

\bibliographystyle{siamplain}
	\bibliography{references}
	
\end{document}

%% file: ex_shared.tex
\usepackage{lipsum}
\usepackage{amsfonts}
\usepackage{graphicx}
\usepackage{epstopdf}
\usepackage{algorithmic}
\ifpdf
  \DeclareGraphicsExtensions{.eps,.pdf,.png,.jpg}
\else
  \DeclareGraphicsExtensions{.eps}
\fi

\newsiamremark{remark}{Remark}
\newsiamremark{hypothesis}{Hypothesis}
\crefname{hypothesis}{Hypothesis}{Hypotheses}
\newsiamthm{claim}{Claim}
\newsiamthm{conj}{Conjecture}
\newsiamremark{Ex}{Example}

\headers{ENERGY STABILITY AND ERROR ESTIMATES FOR GDF DG}{}

\title{Energy Stability and Error Estimates for a Globally Divergence-Free DG Method of Arbitrary Order}

\author{Yuchang Liu\thanks{School of Mathematical Sciences,
         University of Science and Technology of China,
         Hefei, Anhui 230026, P.R. China.
  (\email{lissandra@mail.ustc.edu.cn}).}
\and Yan Jiang\thanks{ School of Mathematical Sciences,
         University of Science and Technology of China, Hefei, 
         Anhui 230026, P.R. China.  
         (\email{jiangy@ustc.edu.cn}).  }
\and Mengping Zhang\thanks{ School of Mathematical Sciences, University of Science and Technology of China, Hefei,
         Anhui 230026, P.R. China.  
         (\email{mpzhang@ustc.edu.cn}).} 
}

\usepackage{amsopn}